\documentclass[12pt]{amsart}

\usepackage{color}
\usepackage{graphicx}
\usepackage[english]{babel}
\usepackage{times}
\usepackage{mathtools}
\usepackage{amsmath}
\usepackage{amsfonts}
\usepackage{amssymb}
\usepackage{amsthm}
\usepackage{transparent}
\usepackage{array}
\usepackage{listings}
\usepackage{enumitem}

\usepackage[margin=1.1in]{geometry}

\usepackage{hyperref}

\hypersetup{
    colorlinks=true,
    citecolor=blue,
    linkcolor=blue,
    filecolor=magenta,
    urlcolor=cyan,
    pdftitle={Overleaf Example},
}

\def\@fnsymbol#1{\ensuremath{\ifcase#1\or *\or \dagger\or \ddagger\or
   \mathsection\or \mathparagraph\or \|\or **\or \dagger\dagger
   \or \ddagger\ddagger \else\@ctrerr\fi}}

\def\cotan{\text{cot}}
\def\eps{\varepsilon}
\def\S{\mathcal{S}}
\def\H{\mathcal{H}}

\def\R{\mathbb{R}}
\def\Z{\mathbb{Z}}
\def\N{\mathbb{N}}
\def\C{\mathbb{C}}

\theoremstyle{plain}
\newtheorem{thm}{Theorem}[section]
\newtheorem{cor}[thm]{Corollary}
\newtheorem{lemma}[thm]{Lemma}
\newtheorem{prop}[thm]{Proposition}

\theoremstyle{definition}
\newtheorem{defo}[thm]{Definition}
\newtheorem{ex}[thm]{Example}
\newtheorem{rem}[thm]{Remark}

\title[Boundedness of the segment multiplier on $L^p(w)$ spaces]{Characterization of the boundedness of the segment multiplier on weighted Lebesgue spaces}

\author[M. F. Barea-Fern\'andez]{Miguel F. Barea-Fern\'andez*}
\address{Miguel F. Barea-Fern\'andez, Departamento de An\'alisis Matem\'atico y Matem\'atica Aplicada, Fa\-cul\-tad de Ciencias Matem\'aticas, Universidad Complutense de Madrid, Plaza de Ciencias 3, 28040 Madrid, Spain.
\texttt{https://orcid.org/0009-0005-5799-5023}}
\curraddr{}
\email{mibarea@ucm.es}
\thanks{*The first author was partially supported by grant PID2024-155917NB-I00, funded by MCIN/AEI/10.13039/501100011033, and by an FPU Grant FPU23/00891, from Ministerio de Ciencia, Innovaci\'on y Universidades (Spain).}

\author[J. Lang]{Jan Lang}
\address{Jan Lang, Department of Mathematics, The Ohio State University, Columbus, OH, United States, Department of Mathematics, Faculty of Electrical Engineering, Czech Technical University in Prague, Czech Republic. \texttt{https://orcid.org/0000-0003-1582-7273}}
\curraddr{}
\email{lang@math.osu.edu}
\thanks{}

\author[J. Soria]{Javier Soria**}
\address{Javier Soria, Departamento de An\'alisis Matem\'atico y Matem\'atica Aplicada, Fa\-cul\-tad de Ciencias Matem\'aticas, Universidad Complutense de Madrid, Plaza de Ciencias 3, 28040 Madrid, Spain and ICMAT. \texttt{https://orcid.org/0000-0003-3098-7056}}
\curraddr{}
\email{javier.soria@ucm.es}
\thanks{**The third author was partially supported by grants PID2024-155917NB-I00 and CEX2019-000904-S, funded by MCIN/AEI/10.13039/501100011033, and Grupo UCM-970966.}


\date{}


\begin{document}

\begin{abstract}
  For \(1<p<\infty\), we characterize the boundedness of the segment multiplier on \(L^p(w)\) by a large-scale Muckenhoupt condition, obtained by restricting the classical \(A_p\) condition to intervals of length at least one. The proof uses a discretization involving the discrete Hilbert transform on an associated weighted sequence space.
\end{abstract}

\maketitle

\section{Introduction}
The theory of Fourier multipliers belongs to the classical core of harmonic analysis. 
In the modern framework it is formulated in terms of the Fourier transform and includes, among others, the Hilbert transform and other singular integral operators. A systematic treatment of Fourier multipliers acting on $L^p$ spaces goes back to H\"ormander~\cite{hormander}.

Let $m \in L^{\infty}(\R^{n})$. 
The multiplier operator $T$ with symbol $m$ is defined, for any Schwartz function $f \in \S(\R^{n})$, by
\[
   \widehat{Tf}(\xi) = m(\xi)\,\hat f(\xi), \qquad \xi \in \R^{n},
\]
where
\[
   \hat f(\xi) = \int_{\R^{n}} f(x)\,e^{-2\pi i x\cdot \xi}\,dx
\]
denotes the Fourier transform of $f$. 
Whenever $T$ admits a bounded extension on a given function space $X$, the corresponding multiplier norm of $m$ over $X$ is the operator norm of $T\colon X \to X$. 
Since the Fourier transform interchanges convolution and pointwise multiplication, one may also write
\[
   Tf = \check m * f,
\]
where $\check m$ is the inverse Fourier transform of $m$ in the distributional sense. 
For basic material on Fourier multipliers we refer, for example, to \cite[Section~2.5]{grafakos}.

In the one-dimensional case we single out two multipliers which will be central in what follows: the \emph{segment multiplier} $S$, given by
\[
   Sf(x) := \int_\R \frac{\sin(x-y)}{x-y}\,f(y)\,dy, 
   \qquad x \in \R,
\]
and the \emph{Hilbert transform} $H$, defined by the principal value
\[
   Hf(x) := \text{p.v.} \int_\R \frac{f(y)}{x-y}\,dy,
   \qquad x \in \R.
\]

These two operators are among the most classical examples of Fourier multipliers. The $L^p$ boundedness of the Hilbert transform was proved by M.\ Riesz~\cite{riesz} for $1 < p < \infty$. The segment multiplier, in turn, could be considered as the one-dimensional version of the ball multiplier, whose boundedness sparked many studies, including the fundamental observation by Fefferman~\cite{fefferman}, where he disproved any $L^p$ boundedness besides the trivial case $p=2$.

The symbols of $S$ and $H$ are
\[
   m_{S}(\xi) = \pi\,\chi_{\big[-\frac{1}{2\pi},\,\frac{1}{2\pi}\big]}(\xi) \quad \text{and} \quad
   m_{H}(\xi) = -i\pi\,\mathrm{sgn}(\xi),
   \quad \xi \in \R.
\]
A straightforward computation with these functions shows that both operators are related, since
\[
   m_{S}(\xi)
   = i\,\frac{m_{H}(\xi + \frac{1}{2\pi}) - m_{H}(\xi - \frac{1}{2\pi})}{2},
   \qquad \xi \in \R.
\]
Hence, whenever $S$ and $H$ are both defined on a function space $X$, the triangle inequality gives
\[
   \|S\|_{X \to X} \leq \|H\|_{X \to X}.
\]

In the scale of the unweighted Lebesgue spaces $L^{p}(\R)$, $1 < p < \infty$, we know more: De Carli and Laeng~\cite{DeCarli} proved that
\begin{equation*} 
    \|H\|_{L^{p} \to L^{p}}
   = \|S\|_{L^{p} \to L^{p}}
   = n_{p}\,\pi,
\end{equation*}
where
\[
   n_{p} :=
   \begin{cases}
      \displaystyle \tan \Bigl(\frac{\pi}{2p}\Bigr), & 1 < p \le 2, \vspace{5pt} \\ 
      \displaystyle \cotan \Bigl(\frac{\pi}{2p}\Bigr), & 2 \le p < \infty.
   \end{cases}
\]

This indicates that the relationship between the Hilbert transform and the segment multiplier might be closer than one might expect at first glance. Beyond the classical Lebesgue spaces, a natural setting in which to study this relationship are the weighted $L^p$ spaces, $L^p(w)$ with norm
$$ \|f\|_{L^p(w)} = \Big( \int_\R |f(x)|^p w(x) \, dx \Big)^{\frac{1}{p}} = \|f w^{\frac{1}{p}}\|_p, $$
where $w\in L^1_{\text{loc}}(\R)$ is an a.e.\ positive function. The boundedness of the Hilbert transform on $L^p(w)$ has been described successfully in this context by Hunt, Muckenhoupt, and Wheeden \cite[Theorem 9]{huntmuckwhee} in terms of the $A_p$ condition on the weight $w$, which also characterizes the weak-type boundedness for $H$ and the boundedness of the maximal operator. Then, our interest now is to investigate the relationship between the boundedness of $S$ and the $A_p$ condition.

The main theorem of our paper, Theorem \ref{thm:main}, states that the boundedness of $S$ is characterized by a truncated version of the $A_p$ condition (see Definition \ref{def:aptrunc}). In Proposition~\ref{ex:ap1counter} we provide a weight that satisfies our condition for every $1 < p < \infty$ but is never an $A_p$ weight so, for any such $p$, $S$ is bounded on $L^p(w)$ but $H$ is not. This shows a big qualitative difference between both operators in the weighted Lebesgue setting, compared to what was proved in \cite{DeCarli} for $L^p(\R)$.

The organization of this paper is as follows. In Section \ref{sec:prel} we cover the basics of the Muckenhoupt $A_p$ condition and some of its fundamental properties, and we also introduce its discrete version and some operators that we will need for our proofs. In Section \ref{sec:apa} we define the truncated $A_p$ condition, we give analogous fundamental properties and explain its relation to both the classical and the discrete $A_p$ conditions. Finally, in Section \ref{sec:seg} we state the main result of this paper, Theorem \ref{thm:main}, which characterizes the boundedness of the segment multiplier on $L^p(w)$ in terms of the weight $w$.

\section{Preliminary definitions and results} \label{sec:prel}

If $E \subseteq \R$, we will denote by $\chi_E$ the characteristic function of the set $E$, and the Lebesgue measure of $E$ as $|E|$. For any $1 \leq p \leq \infty$, denote its conjugate exponent by $p' = \frac{p}{p-1}$, with the usual meaning of $1' = \infty$ and $\infty' = 1$.

Whenever we write $ f(s) \lesssim g(s) $, the expression denotes an inequality like $ f(s) \leq C g(s) $, where $C$ is a constant independent of any variables. Analogously, the notation $ f(s) \approx g(s) $ means the same as having both $ f(s) \lesssim g(s) $ and $ g(s) \lesssim f(s) $.

If $I$ is any interval and $\lambda > 0$, we denote by $\lambda I$ the interval with the same center as $I$ and $\lambda$ times its length.

\begin{defo}
    For $\lambda\in\R$, denote by $T_\lambda$ the translation operator that acts as $T_\lambda f(x) = f(x-\lambda).$ For $a > 0$, denote by $E_a$ the dilation operator that acts as $E_a f(x) = f(ax)$.
\end{defo}

\subsection{The Muckenhoupt condition}

We begin by recalling the definition and main properties of the classical $A_p$ weights, as written in \cite[Section 7.1]{grafakos}. Results in more general measure spaces can be found in \cite{stromtorch}.

\begin{defo}
    Let $1 < p < \infty$, and let $w$ be a weight on $\R$; that is, a non-negative locally integrable function. We say that $w$ satisfies the $A_p$ condition if we have
    $$\sup_I \left( \frac{1}{|I|} \int_I w(x) \, dx \right) \left(\frac{1}{|I|} \int_I w(x)^{-\frac{1}{p-1}}\, dx \right)^{p-1} < \infty, $$
    where the supremum is taken over all bounded intervals $I\subseteq\R$. We denote the previous supremum by $[w]_{A_p}$, and we say that $w\in A_p$.
\end{defo}

\begin{rem}
    Given a weight $w$ on $\R$ and $E \subseteq \R$, we will denote $w(E) = \int_E w(x) \, dx$.
\end{rem}

From the definition of the $A_p$ class, the following basic properties follow:

\begin{prop}[{\cite[Proposition 7.1.5]{grafakos}}] \label{elementaryap}
    Let $1 < p < \infty$ and $w\in A_p$. Then:
    \begin{enumerate}
        \item $[E_a w]_{A_p} = [w]_{A_p} $ and $[aw]_{A_p} = [w]_{A_p} $, for any $a > 0$.
        \item $[T_\lambda w]_{A_p} = [w]_{A_p} $, for any $\lambda \in \R$.
        \item $w^{-\frac{1}{p-1}} \in A_{p'}$, with constant $ [w^{-\frac{1}{p-1}}]_{A_{p'}} = [w]_{A_p}^{\frac{1}{p-1}}. $
        \item $[w]_{A_p} \geq 1$, and equality holds if and only if $w$ is constant.
        \item If $1 < p < q < \infty$, then $[w]_{A_q} \leq [w]_{A_p} $, so the classes $A_p$ are increasing in $p$.
        \item The $A_p$ constant coincides with the expression
        $$ [w]_{A_p} = \sup_{I } \, \sup\left\{ \frac{\big(\frac{1}{|I|} \int_I |f(x)|\, dx \big)^p}{\frac{1}{w(I)} \int_I |f(x)|^p w(x) \, dx} \, : \, 0 < \int_I |f(x)|^p w(x) \, dx < \infty\right\}. $$
        \item The measure $w(x) dx$ is doubling: for any $a > 1$ and any interval $I$, we have that $w(a I) \leq a^{p} [w]_{A_p} w(I)$.
    \end{enumerate}
\end{prop}

The most important property of the $A_p$ condition is given by the fact that it characterizes the boundedness of the Hardy-Littlewood maximal operator.

\begin{defo}
    For any function $f\in L^1_{\text{loc}}(\R)$, define the Hardy-Littlewood maximal operator as 
    $$ Mf(x) = \sup_{x\in I}\Big\{ \frac{1}{|I|} \int_I |f(y)| \, dy \Big\}, $$
    where the supremum is taken on all intervals $I \subseteq \R$ containing the point $x$.
\end{defo}

\begin{thm}[{\cite[Theorem 9]{muck}}]
    Let $w$ be a weight on $\R$ and $1 < p < \infty$. Then, $M$ is bounded on $L^p(w)$ if and only if $w\in A_p$.
\end{thm}

As we stated in the introduction, the same result holds for the Hilbert transform in this range of exponents:

\begin{thm}[{\cite[Theorem 9]{huntmuckwhee}}]
    Let $w$ be a weight on $\R$ and $1 < p < \infty$. Then, $H$ is bounded on $L^p(w)$ if and only if $w\in A_p$.
\end{thm}

\subsection{The discrete Muckenhoupt condition}

On discrete weighted spaces, we can define an analogous class of $A_p$ weights with similar properties.

\begin{defo}
    Let $w = \{w_k\}_{k\in\Z}$ be a weight on $\Z$. We say that $w$ satisfies the discrete $A_p^d$ condition, $w\in A_p^d$, if
    \begin{equation} \label{apd}
        [w]_{A_p^d} = \sup_J \left( \frac{1}{\# J} \sum_{k\in J} w_k \right) \left(\frac{1}{\# J} \sum_{k\in J} w_k^{-\frac{1}{p-1}} \right)^{p-1} < \infty,
    \end{equation}
    where the supremum is taken over all finite intervals $J\subseteq\Z$, and $\# J$ denotes the cardinality of the set $J$. We can also write this condition as
    $$[w]_{A_p^d} = \sup_{m\leq l} \left( \frac{1}{l-m+1} \sum_{k=m}^l w_k \right) \left(\frac{1}{l-m+1} \sum_{k=m}^l w_k^{-\frac{1}{p-1}} \right)^{p-1} < \infty.$$
\end{defo}

Many properties are shared with the continuous setting, and they can be derived by repeating the proofs or by observing that, for a discrete weight $w = \{w_k\}_k$,
$$ [w]_{A_p^d} \approx [\Tilde{w}]_{A_p}, \quad \text{where} \quad \Tilde{w}(x) = \sum_{k\in\Z} w_k \chi_{[k,k+1)}(x). $$
We also refer back to the general development in \cite{stromtorch}.

Hence, for the sake of brevity, we will only state the result we will need regarding the discrete Hilbert transform.

\begin{defo}
    Let $a = (a_k)_{k\in\Z}$ be a sequence. Its discrete Hilbert transform is the sequence $\H a$ with general term
    $$ (\H a)_k = \sum_{j\neq k} \frac{a_j}{k-j}, $$
    if the series converges.
\end{defo}

\begin{thm}[{\cite[Theorem 10]{huntmuckwhee}}] 
    Let $1 < p < \infty$ and let $w = \{w_k\}_{k\in\Z}$ be a discrete weight. Then, $\H$ is bounded on $\ell^p(w)$ if and only if $w\in A_p^d$.
\end{thm}

\subsection{Some useful operators}

We also introduce some operators that will be necessary in the proofs of the following results.

\begin{defo} 
    For $s>0$, define the moving average operators:
    
    \begin{equation}\label{movingavg}
        \begin{split}
            &A_sf(x) = \frac{1}{s} \int_x^{x+s} f(t) \, dt,  \qquad  A_{-s}f(x) = \frac{1}{s} \int^x_{x-s} f(t) \, dt, \\
            &B_s f(x) = \frac{1}{2s} \int^{x+s}_{x-s} f(t) \, dt = \frac{A_sf(x) + A_{-s}f(x)}{2}.
        \end{split}
    \end{equation}
\end{defo}

\begin{rem}
    All three operators in (\ref{movingavg}) are of convolution type, and their kernels are, respectively:
    \begin{align*}
        a_s(t) &= \frac{1}{s} \chi_{(-s,0)}(t), \quad a_{-s}(t) = \frac{1}{s} \chi_{(0,s)}(t), \quad \text{and} \quad b_s(t) = \frac{1}{2s} \chi_{(-s,s)}(t).
    \end{align*}
\end{rem}

We will prove some properties about these operators in general function spaces.

\begin{defo}
    Let $X$ be a normed space of functions. We say that its norm is monotone if $|f| \leq |g|$ a.e.\ implies $\|f\|_X \leq \|g\|_X$.
    
    We say that the norm has the Fatou property if the following is satisfied: for any sequence $\{f_n\}_{n\in\N} \subseteq X$ such that $|f_n| \nearrow |f|$ a.e., we have $\|f_n\|_X \nearrow \|f\|_X$.
\end{defo}

\begin{rem}
    Given any weight $w$ and $1 < p < \infty$, the norm of $L^p(w)$ is monotone and has the Fatou property.
\end{rem}

\begin{prop}
    Let $X$ be a Banach space of functions with monotone norm. If $B_s$ is bounded on $X$ for every $s>0$, then for any bounded measurable set $E\subseteq \R$ there exists a constant $C_E > 0$ such that
    $$\|\chi_E * f\|_{X} \leq C_E \|f\|_{X};$$
    that is, the operator given by $ f \mapsto \chi_E * f $ is bounded on $X$.
\end{prop}

\begin{proof}
    Let $f\in X$ and suppose that $E \subseteq I = (-R,R)$. Taking the absolute value, we find that
    $$|\chi_E * f| \leq \chi_E*|f| \leq \chi_I*|f| = 2R B_R|f|. $$
    Then, $\|\chi_E*f\|_{X} \leq 2R \|B_R\|_{X} \|f\|_{X},$ where $\|B_R\|_{X}$ denotes the operator norm of $B_R$.
\end{proof}

\begin{cor} \label{bbs}
    Let $X$ be a Banach space of functions with monotone norm and suppose that $B_s$ is bounded on $X$ for each $s>0$. Then, if $f\in L^\infty(\R)$ with bounded support, the convolution with $f$ is bounded on $X$.
\end{cor}

\begin{cor}
    Let $X$ be a Banach space of functions with monotone norm and let $s>0$. Then, $B_s$ is bounded on $X$ if and only if both $A_s$ and $A_{-s}$ are bounded on $X$.
\end{cor}

In the proof of Theorem \ref{thm:main}, we will frequently use the Poisson operator, whose boundedness can be derived from that of the averaging operators in (\ref{movingavg}). This is a result that also appears in \cite{bkls} but, for completeness, we recall it here.

\begin{defo} \label{def:poisson}
    We denote by $P$ the convolution operator with respect to the Poisson kernel, $\frac{1}{x^2 + 1}$:
    $$ Pf(x) = \int_\R \frac{f(y)}{1+ (x-y)^2} \, dy. $$
\end{defo}

\begin{lemma}\label{poisson}
    Let $X$ be a Banach space of functions, whose norm $\|\cdot\|_X$ is monotone and has the Fatou property. Suppose that the operators $\{B_s\}_{s\geq 1}$ are uniformly bounded on $X$, $\|B_s\|\leq C$, with $C$ independent of $s\geq 1$. Then, $P$ is bounded on $X$.
\end{lemma}

\begin{proof}
    For $j\in \N$, let $b_j = \frac{\chi_{(-j,j)}}{2j}$ the convolution kernel of $B_j$, and define $g_n = \sum_{j=1}^n b_j/j^2. $ Since $\{B_j\}_j$ are uniformly bounded, the operator $G_n f = g_n * f = \sum_{j=1}^n B_j f /j^2$ is bounded on $X$. Moreover, $\{g_n*f\}_n$ is a Cauchy sequence: if $m > n$, the difference in their corresponding terms is $G_m f - G_n f = \sum_{j=n+1}^m B_j f/j^2$, so
    $$ \|G_m f - G_n f\|_X \leq \sum_{j=n+1}^m \frac{1}{j^2} \|B_j f\|_X \leq C\|f\|_X \sum_{j=n+1}^m \frac{1}{j^2}. $$
    Then, we can define the operator $Gf := \lim_n g_n*f$. Because of the Fatou property, we get that
    $$\|Gf\|_X \leq \liminf_n \|g_n * f\|_X \leq C \sum_{j=1}^\infty \frac{1}{j^2} \|f\|_X, $$
    so $G$ is bounded. Using the monotone convergence theorem, we can see that $G$ is the convolution operator with kernel $g = \sum_{j=1}^\infty \chi_{(-j,j)}/(2j^3). $

    Note that for $k\geq 1$ we can compare $\sum_{j=k}^\infty \frac{1}{j^3} \approx \frac{1}{k^2 + 1}$. Then, if $|x| \in (k,k+1)$, we can see that
    $$g(x) = \sum_{j=k+1}^\infty \frac{1}{2j^3} \approx \frac{1}{k^2 + 1} \approx \frac{1}{x^2 + 1}.$$

    Because both $G$ and $P$ have positive kernels, we have that
    $$ |Gf| \leq G(|f|), \quad |Pf| \leq P(|f|), $$
    and also, for every non-negative function $f$, $ Pf \approx Gf. $ Hence, we obtain that
    $$ \|Pf\|_X \leq \|P(|f|)\|_X \approx \|G(|f|)\|_X \lesssim \|f\|_X, $$
    so $P$ is a bounded operator on $X$.
\end{proof}

\section{The truncated Muckenhoupt condition} \label{sec:apa}

During this work we only consider weighted $L^p(w)$ spaces over $\R$. We now introduce the following condition:

\begin{defo} \label{def:aptrunc}
    Let $w$ be a weight on $\R$. We say that $w$ satisfies the \emph{truncated} $A_p$ condition (for some level $a>0$) if we have
    $$\sup_{|I|\geq a} \left( \frac{1}{|I|} \int_I w(x) \, dx \right) \left(\frac{1}{|I|} \int_I w(x)^{-\frac{1}{p-1}}\, dx \right)^{p-1} < \infty, $$
    where the supremum is taken over all bounded intervals $I\subseteq\R$ with length $|I| \geq a$. We denote the previous supremum by $[w]_{A_{p,a}}$, and we say that $w \in A_{p,a}$.
\end{defo}

As the following proposition shows, the threshold $a$ does not have qualitative relevance:

\begin{prop}
    For every $a>0$, all $A_{p,a}$ classes coincide: $A_{p,a} = A_{p,1}. $ Moreover, for $w\in A_{p,1}$, we have that $ [w]_{A_{p,2a}} \leq [w]_{A_{p,a}} \leq 2^p [w]_{A_{p,2a}}. $
\end{prop}

\begin{proof}
    From the definition, it is clear that $ [w]_{A_{p,2a}} \leq [w]_{A_{p,a}} $.

    To prove the other inequality, consider a weight $w\in A_{p,2a}$ and note that Definition~\ref{def:aptrunc} is equivalent to having, for each interval $I$ with $|I| \geq 2a$,
    $$ \bigg( \int_I w(x) dx \bigg) \left( \int_{I} w(x)^{-\frac{1}{p-1}} dx \right)^{p-1} \leq [w]_{A_{p,2a}} |I|^p. $$
    Then, if $J = (c,d)$ is an interval with $a \leq |J| \leq 2a$, we have that
    \begin{align*}
        \bigg( \int_J w(x) dx \bigg) \left( \int_J w(x)^{-\frac{1}{p-1}} dx \right)^{p-1} &\leq \bigg( \int_c^{c+2a} w(x) dx \bigg) \left( \int_c^{c+2a} w(x)^{-\frac{1}{p-1}} dx \right)^{p-1} \\
        &\leq [w]_{A_{p,2a}} 2^p a^p \leq 2^p [w]_{A_{p,2a}} |J|^p.
    \end{align*}
    Thus, $w\in A_{p,a}$ and $ [w]_{A_{p,a}} \leq 2^p [w]_{A_{p,2a}} $. In particular, $A_{p,a} = A_{p,2a}$. Because for $a < b$, $A_{p,a} \subseteq A_{p,b} $, a quick argument shows that $A_{p,a} = A_{p,1}$ for any $a > 0$.
\end{proof}

The truncated $A_{p,1}$ weights satisfy many of the elementary properties of the $A_p$ class.

\begin{prop}
    Let $1 < p < \infty$, $a>0$, and $w\in A_{p,1}$. Then:
    \begin{enumerate}
        \item $[E_b w]_{A_{p,a}} = [w]_{A_{p,ba}} $ and $[bw]_{A_{p,a}} = [w]_{A_{p,a}} $, for any $b > 0$.
        \item $[T_\lambda w]_{A_{p,a}} = [w]_{A_{p,a}} $, for any $\lambda \in \R$.
        \item $w^{-\frac{1}{p-1}} \in A_{p',1}$, with constant $ [w^{-\frac{1}{p-1}}]_{A_{p',a}} = [w]_{A_{p,a}}^{\frac{1}{p-1}}. $
        \item $[w]_{A_{p,a}} \geq 1$, and equality holds if and only if $w$ is constant.
        \item If $1 < p < q < \infty$, then $[w]_{A_{q,a}} \leq [w]_{A_{p,a}} $, so the classes $A_{p,1}$ are increasing in $p$.
        \item The $A_{p,a}$ constant coincides with the expression
        $$ [w]_{A_{p,a}} = \sup_{|I| \geq a } \, \sup\left\{ \frac{\big(\frac{1}{|I|} \int_I |f(x)|\, dx \big)^p}{\frac{1}{w(I)} \int_I |f(x)|^p w(x) \, dx} \, : \, 0 < \int_I |f(x)|^p w(x) \, dx < \infty\right\}. $$
        \item For any $b > 1$ and any interval $I$, we have that $w(b I) \leq b^{p} [w]_{A_{p,|I|}} w(I)$.
    \end{enumerate}
\end{prop}

The proofs of these properties are not significantly different from Proposition \ref{elementaryap}, so we refer to \cite[Section 7.1]{grafakos}. The same is true for the following characterization regarding an adequate maximal operator.

\begin{defo}
    Let $a>0$. For $f\in L^1_{\text{loc}}(\R)$, define the truncated maximal operator $M_a$ as
    $$ M_a f(x) = \sup_{x \in I, \, |I| \geq a} \Big\{ \frac{1}{|I|} \int_I |f(y)|\, dy\Big\}. $$
\end{defo}

\begin{prop}
    Let $w$ be a weight on $\R$ and $1 < p < \infty$. Then, $M_a$ is bounded on $L^p(w)$ if and only if $w\in A_{p,1}$.
\end{prop}

More properties can be derived in a similar manner for the truncated case. We will also prove another characterization in relation to the operators $B_s$ introduced in (\ref{movingavg}).

\begin{prop} \label{unifbdd}
    Let $1 < p < \infty$, and let $w$ be a weight. Then, $w \in A_{p,1}$ if and only if the family of operators $\{B_s\}_{s \geq 1}$ is uniformly bounded on $L^p(w)$.
\end{prop}

To prove this, we will make use of the following theorem by Heinig and Sinnamon:

\begin{thm}[{\cite[Theorem 2.2]{heinsin}}]\label{heinsin}
    Let $u$ and $v$ be weights on $\R$, and $1 < p \leq q < \infty$. Let $a,b: \R \to \R$ be increasing, differentiable functions such that $a(-\infty) = b(-\infty) = -\infty$, $a(\infty) = b(\infty) = \infty $, and $a(x) < b(x)$ for all $x \in \R$. For $f:\R \to \C$, define
    $$ Tf(x) = \int_{a(x)}^{b(x)} f(y) \, dy. $$

    Then, $T: L^p(u) \to L^q(v)$ if and only if
    $$ K = \sup \bigg\{ \bigg(\int_{a(x)}^{b(t)} u(y)^{-\frac{1}{p-1}} \, dy \bigg)^{\frac{p-1}{p}} \bigg(\int_t^x v(y) \, dy\bigg)^{\frac{1}{q}} \, : \, t \leq x, \; a(x) \leq b(t) \bigg\} < \infty. $$
    Moreover, $ K \leq \|T\| \leq 2 p^{1/q} (p')^{1/p'} K $.
\end{thm}

\begin{rem}
    The theorem was originally stated for spaces over $(0,\infty)$. Doing an exponential change of variables, we obtain exactly the same statement for spaces over $\R$.
\end{rem}

\begin{proof}[Proof (Proposition \ref{unifbdd}).]
    We will denote
    $$ \langle w \rangle_s = \sup_{|I| = s} \bigg\{ \bigg(\int_I w(x) \, dx \bigg) \bigg(\int_I w(x)^{-\frac{1}{p-1}} \, dx \bigg)^{p-1} \bigg\}. $$
    
    Using Theorem \ref{heinsin} for $a(x) = x-s$, $b(x) = x+s$ and $u = v = w$, we obtain that $B_s$ is bounded if and only if
    $$K = \sup \bigg\{ \bigg(\int_{x-s}^{t+s} w(y)^{-\frac{1}{p-1}} \, dy \bigg)^{\frac{p-1}{p}} \bigg(\int_t^x w(y) \, dy\bigg)^{\frac{1}{p}} \, : \, t \leq x \leq t + 2s \bigg\} < \infty.$$

    We can bound this expression as
    \begin{align*}
        K & \leq \sup_t \bigg\{ \bigg(\int_{t-s}^{t+s} w(y)^{-\frac{1}{p-1}} \, dy \bigg)^{\frac{p-1}{p}} \bigg(\int_t^{t+2s} w(y) \, dy\bigg)^{\frac{1}{p}} \bigg\} \\
        & \leq \sup_t \bigg\{ \bigg(\int_{t-s}^{t+2s} w(y)^{-\frac{1}{p-1}} \, dy \bigg)^{\frac{p-1}{p}} \bigg(\int_{t-s}^{t+2s} w(y) \, dy\bigg)^{\frac{1}{p}} \bigg\} = \langle w \rangle_{3s}^{\frac{1}{p}}.
    \end{align*}
    To give a lower bound, considering the case $x = t+s$, we have that
    \begin{align*}
        K & \geq \sup_t \bigg\{ \bigg(\int_{t}^{t+s} w(y)^{-\frac{1}{p-1}} \, dy \bigg)^{\frac{p-1}{p}} \bigg(\int_t^{t+s} w(y) \, dy\bigg)^{\frac{1}{p}} \bigg\} = \langle w \rangle_s^{\frac{1}{p}}.
    \end{align*}
    Since we know that $\|B_s\| \approx K/2s$, we have found
    that
    $$ (2s)^{-p} \langle w \rangle_s \leq \|B_s\|^p \lesssim (6s)^{-p} \langle w \rangle_{3s}. $$

    With these inequalities, it is straightforward that
    $$ [w]_{A_{p,1}} = \sup_{s > 1} \, s^{-p} \langle w \rangle_s \leq (\sup_{s>1} \|B_s\|)^p \lesssim \sup_{s > 1} \, (3s)^{-p} \langle w \rangle_{3s} = [w]_{A_{p,3}} \approx [w]_{A_{p,1}}, $$
    and the result follows.
\end{proof}

\begin{rem}
    Note that we also have
    $$ [w]_{A_{p}} = \sup_{s > 0} \, s^{-p} \langle w \rangle_s \leq (\sup_{s>0} \|B_s\|)^p \lesssim \sup_{s > 0} \, (3s)^{-p} \langle w \rangle_{3s} = [w]_{A_{p}}, $$
    so it is true that $w \in A_p$ if and only if the operators $\{B_s\}_{s>0}$ are uniformly bounded.
\end{rem}

\begin{rem}
    A similar result was proved in \cite[Theorem 2.1]{berezhnoi} and \cite[Proposition 3.3]{nieraeth}. For a family of pairwise disjoint intervals $\textbf{I} = \{I_j\}_j$, define the following averaging operator
    $$ A_\textbf{I}f(x) = \sum_j \bigg( \frac{1}{|I_j|} \int_{I_j} f(y) \, dy \bigg) \chi_{I_j}(x). $$
    Then, in the particular case of $L^p(w)$ over $\R$, the theorem states that $w \in A_p$ if and only if the operators $A_\textbf{I}$ are bounded uniformly over $\textbf{I}$.
\end{rem}

A natural question to ask is whether the truncated classes coincide with the classical ones. The following proposition shows that the inclusions are strict.

\begin{prop} \label{ex:ap1counter}
    For any $1 < p < \infty$, $A_p \subsetneq A_{p,1}$.
\end{prop}

\begin{proof} 
    Let $\delta \in (0, \frac{1}{2})$, and take $M\in \N$ such that 
    $$ \sum_{k= M+1}^\infty k 2^{-k} < \delta \qquad \text{and} \qquad \sum_{k= M+1}^\infty k^{\frac{1}{p-1}} 2^{-k} < \delta. $$
    For $k\geq M+1$, denote the intervals $ I_k^0 = (2^{-k-1}, 2^{-k - \frac{1}{2}})$ and $ I_k^1 = [2^{-k-\frac{1}{2}}, 2^{-k})$, and their union $ I_k = (2^{-k-1}, 2^{-k})$. With these, define the weight $w$ at any point $x\in \R$ as
    $$ w(x) = \begin{cases}
        k & \quad x\in I_k^0, \\
        \frac{1}{k} & \quad x\in I_k^1, \\
        1 & \quad x\not \in \bigcup_{k\geq M+1} I_k.
    \end{cases} $$
    Let us check that this weight satisfies $w\in A_{p,1}\setminus A_p$, for each $1 < p < \infty$.

    To see that it is not an $A_p$ weight, we will consider the averages over the intervals $I_k$: since $|I_k^0| \approx |I_k| \approx |I_k^1|$, a quick computation gives
    $$ \frac{1}{|I_k|} \int_{I_k} w(x) \, dx \approx k \qquad \text{and} \qquad \frac{1}{|I_k|} \int_{I_k} w(x)^{-\frac{1}{p-1}} \, dx \approx k^{\frac{1}{p-1}}.
    $$
    Hence, we have that
    $$ \bigg( \frac{1}{|I_k|} \int_{I_k} w(x) \, dx \bigg) \bigg( \frac{1}{|I_k|} \int_{I_k} w(x)^{-\frac{1}{p-1}} \,dx \bigg)^{p-1} \approx k^2 \underset{k\to \infty}{\longrightarrow} \infty, $$
    so the weight $w$ does not belong to $A_p$. Note that $|I_k| \to 0$ when $k \to \infty$.

    Denote $I = \bigcup_{k\geq M+1} I_k$, and observe that $|I| < \delta$. We also have
    $$ w(I) := \int_I w(x) \, dx = \sum_{k= M+1}^\infty \int_{I_k} w(x) \,dx \approx \sum_{k= M+1}^\infty k 2^{-k} < \delta, $$
    $$ w^{-\frac{1}{p-1}}(I) := \int_I w(x)^{-\frac{1}{p-1}} \, dx = \sum_{k= M+1}^\infty \int_{I_k} w(x)^{-\frac{1}{p-1}} \,dx \approx \sum_{k= M+1}^\infty k^{\frac{1}{p-1}} 2^{-k} < \delta. $$

    Take an interval $E \subseteq \R$ with $|E| > 2\delta$, which implies that $|E\setminus I| \approx |E| $. With this, we can bound the averages over $E$ as follows:
    \begin{align*}
        \frac{1}{|E|} \int_E w(x) \,dx &\leq \frac{1}{|E|} \int_{E\cup I} w(x) \,dx = \frac{1}{|E|} (w(I) + |E\setminus I|) \leq \frac{1}{|E|} (\delta + |E\setminus I|) \lesssim 1.
    \end{align*}
    Similarly, taking the average on $w^{-\frac{1}{p-1}}$, we obtain
    \begin{align*}
        \frac{1}{|E|} \int_E w(x)^{-\frac{1}{p-1}} \, dx &\leq \frac{1}{|E|} (w^{-\frac{1}{p-1}}(I) + |E\setminus I|) \leq \frac{1}{|E|} (\delta + |E\setminus I|) \lesssim 1.
    \end{align*}
    Therefore, we prove that $w \in A_{p,2\delta} = A_{p,1}$, since
    $$ \bigg( \frac{1}{|E|} \int_{E} w(x) \,dx \bigg) \bigg( \frac{1}{|E|} \int_{E} w(x)^{-\frac{1}{p-1}} \,dx \bigg)^{p-1} \lesssim 1. $$

\end{proof}

\begin{rem}
    The weight in the proof of Proposition \ref{ex:ap1counter} exhibits a bad, oscillatory behavior concentrated at the origin, which makes it fail the $A_p$ condition. However, this localized phenomenon does not impact as much the $A_{p,1}$ condition. This fact highlights the strictly large-scale nature of $A_{p,1}$ in contrast to the global one of $A_p$.
\end{rem}

The truncated $A_{p,1}$ condition also has natural relations to discrete $A_p^d$ weights by means of certain averaging operators. Next, we show some of these relations, as they will be necessary in Section \ref{sec:seg}.

\begin{prop} \label{discrap}
    Let $w$ be a weight on $\R$, $1 < p < \infty$, and suppose that $w\in A_{p,1}$. If we define
    $$w_k = \int_{ak}^{a(k+1)} w(x) \, dx $$
    for some $a>0$ and for each $k\in\Z$, then the discrete weight $w_d = \{w_k\}_k$ satisfies the $A_p^d$ condition.
\end{prop}

\begin{proof}
    Let $k\in \Z$. Note that the function $t\mapsto t^\theta$ on $\R^+$ is convex for any $\theta < 0$, so we can apply Jensen's inequality with $\theta = -\frac{1}{p-1}$ to get that
    $$\left(\frac{1}{a}\int_{ak}^{a(k+1)} w(x)\, dx \right)^{-\frac{1}{p-1}} \leq \frac{1}{a} \int_{ak}^{a(k+1)} w(x)^{-\frac{1}{p-1}}\, dx. $$

    For any interval $J = [m,l]\subseteq\Z$, define the real interval $\Tilde{J} = [ma,(l+1)a]\subseteq\R$. Then, it is clear that we have
    \begin{equation} \label{1term}
        \frac{1}{\# J} \sum_{k\in J} w_k = \frac{1}{\# J} \sum_{k\in J} \int_{ak}^{a(k+1)} w(x) \, dx = \frac{a}{|\Tilde{J}|} \int_{\Tilde{J}} w(x) \, dx. 
    \end{equation}
    Looking at the other term in the supremum in (\ref{apd}), we get
    \begin{equation} \label{2term}
    \begin{split}
        \left(\frac{1}{\# J} \sum_{k\in J} w_k^{-\frac{1}{p-1}} \right)^{p-1} &= \left(\frac{1}{\# J} \sum_{k\in J} \left( \int_{ak}^{a(k+1)} w(x) \, dx \right)^{-\frac{1}{p-1}} \right)^{p-1} \\
        &\leq \left(\frac{a^{-p'}}{\# J} \sum_{k\in J}  \int_{ak}^{a(k+1)} w(x)^{-\frac{1}{p-1}} \, dx \right)^{p-1} \\
        &= \frac{1}{a}\left(\frac{1}{|\Tilde{J}|}  \int_{\Tilde{J}} w(x)^{-\frac{1}{p-1}} \, dx \right)^{p-1}.
    \end{split}
    \end{equation}
    This way, we can bound the product of the left-hand sides of (\ref{1term}) and (\ref{2term}) to get:
    \begin{align*}
        \left( \frac{1}{\# J} \sum_{k\in J} w_k \right)\left(\frac{1}{\# J} \sum_{k\in J} w_k^{-\frac{1}{p-1}} \right)^{p-1} \leq \left( \frac{1}{|\Tilde{J}|} \int_{\Tilde{J}} w(x) \, dx\right)\left(\frac{1}{|\Tilde{J}|}  \int_{\Tilde{J}} w(x)^{-\frac{1}{p-1}} \, dx \right)^{p-1},
    \end{align*}
    so taking the supremum over all intervals $J\subseteq\Z$, we see that $w_d$ has the discrete $A_p^d$ condition with constant at most $[w]_{A_{p,a}}$.
\end{proof}

\begin{prop} 
    Let $1 < p < \infty$ and $w\in A_{p,1}$. Given $a > 0$, let $w_d = \{w_k\}_k$ be the discrete weight where $w_k = \int_{ak}^{a(k+1)} w(x) \, dx $ for every $k\in\Z$. Define the operator $V = V_a: L^p(w) \to \ell^p(w_d)$, as
    $$Vf_k = \int_{ak}^{a(k+1)} f(x) \, dx. $$
    Then, $V$ is well-defined, bounded and surjective.
\end{prop}

\begin{proof}
    Using H\"older's inequality, we obtain
    \begingroup
    \allowdisplaybreaks
    \begin{align*}
        \sum_{k\in\Z} |Vf_k|^p w_k &= \sum_{k\in\Z} \bigg|\int_{ak}^{a(k+1)} f(x)w_k^{\frac{1}{p}} \, dx\bigg|^p = \sum_{k\in\Z} \bigg|\int_{ak}^{a(k+1)} f(x)w(x)^{\frac{1}{p}}\left(\frac{w_k}{w(x)}\right)^{\frac{1}{p}} \, dx\bigg|^p \\
        &\leq \sum_{k\in\Z} \bigg(\int_{ak}^{a(k+1)} |f(x)|^p w(x) \, dx\bigg) \bigg(\int_{ak}^{a(k+1)} \left(\frac{w_k}{w(x)}\right)^{\frac{p'}{p}} \, dx\bigg)^\frac{p}{p'} \\
        &= \sum_{k\in\Z} \bigg(\int_{ak}^{a(k+1)} |f(x)|^p w(x) \, dx\bigg) \bigg( \int_{ak}^{a(k+1)} w(x) \, dx \bigg) \bigg(\int_{ak}^{a(k+1)} w(x)^{-\frac{1}{p-1}} \, dx\bigg)^{p-1} \\
        &\leq a^p [w]_{A_{p,a}} \sum_{k\in\Z} \bigg(\int_{ak}^{a(k+1)} |f(x)|^p w(x) \, dx\bigg) = a^p [w]_{A_{p,a}} \int_\R |f(x)|^p w(x) \, dx.
    \end{align*}
    \endgroup
    Thus, we see that $Vf \in \ell^p(w_d)$ and that $\|V\| \leq a [w]_{A_{p,a}}^{\frac{1}{p}}$. The surjectivity of $V$ is immediate.
\end{proof}

It is natural to ask whether the reverse of Proposition \ref{discrap} is true. Namely, if having the $A_p^d$ condition on $w_d$ implies that $w \in A_{p,1}$. The following example disproves this:

\begin{ex}
    Let $1 < p < \infty$, and let
    $$ w(x) = \begin{cases}
        x^{p-1}, & 0 < x < 1,\\
        1, & x \not \in (0,1).
    \end{cases} $$
    Note that, for $x \in (0,1)$, $w(x)^{-\frac{1}{p-1}} = \frac{1}{x}$, so $w^{-\frac{1}{p-1}} \not \in L^1_{\text{loc}}$. Therefore, $w \not \in A_{p,1}$.

    On the other hand, a direct computation gives $w_k \approx 1$, so $w_d$ is an $A_p^d$ weight.
\end{ex}

\section{The segment multiplier on weighted Lebesgue spaces} \label{sec:seg}

Our main goal is to characterize the boundedness of the segment multiplier. For this purpose, we recall its definition as well as the truncated Hilbert transform's:

\begin{defo}
    Given $\eps > 0$, define the following kernel functions:
    $$ \mathrm{sinc}(t) = \frac{\sin t}{t}, \qquad h_\eps(t) = \frac{1}{t}\chi_{\R\setminus(-\eps,\eps)}. $$

    Their convolution operators are, respectively, the segment multiplier and the truncated Hilbert transform:
    $$ Sf = \mathrm{sinc} *f, \qquad H_\eps f = h_\eps * f. $$ 
\end{defo}

\begin{rem}
    Let $w$ be a weight and $1 < p < \infty$. Suppose that $B_s$ is bounded for any $s>0$. Then, given $\eps,\delta>0$, Corollary \ref{bbs} implies that $H_\eps - H_\delta$ is bounded on $L^p(w)$. Thus, the boundedness of $H_\eps$ and $H_\delta$ are equivalent.
\end{rem}

We are now able to prove our main result:

\begin{thm} \label{thm:main}
    Let $1 < p < \infty$ and let $w$ be a weight. Then, $S:L^p(w) \to L^p(w)$ is bounded if and only if $w\in A_{p,1}$.
\end{thm}

\begin{proof}
    Assume that $S$ is a bounded operator on $L^p(w)$. For each $k\in\Z$ and $j\in\{1,2\}$, denote the intervals
    $$I_k = \left[ \frac{\pi}{4}k,\frac{\pi}{4}(k+1)  \right] \qquad \text{and} \qquad I_k^j = \left[ \frac{\pi}{8} (2k+j-1),\frac{\pi}{8}(2k+j)  \right].$$
    Note that $I_k^j$ is the $j$-th half of the interval $I_k$. Denote also $I_{k,l} = \bigcup_{m=k}^{k+l-1} I_m = \frac{\pi}{4} \cdot \left[ k, k+l  \right], $ for $l\geq 1$, and the intersection $I_{k,l}^j = I_{k,l} \cap \left( \bigcup_{m=0}^\infty I_{k+8m}^j \right). $

    Let us prove, for each $k\in \Z$ and $l\in \N\setminus 4\N$, that
    \begin{equation} \label{eql}
        \int_{I_{k+l,l}} w(x) \, dx \approx \int_{I_{k,l}} w(x)\, dx.
    \end{equation}
    
    For $l = 1$ we have $I_{k,1} = I_k$ and $I_{k,1}^j = I_k^j$. Let $x\in I_{k+1}^j$, and note that for each $y\in I_k^j$ we have $x-y \in [\frac{\pi}{8}, \frac{3\pi}{8}].$ With this, we can obtain the following estimate
    \begin{align*}
        |S\chi_{I_k^j}(x)| = \bigg| \int_{I_k^j} \frac{\sin(x-y)}{x-y} \, dy \bigg| \geq \frac{\sin\left(\frac{\pi}{8}\right)}{\frac{3\pi}{8}} |I_k^j| = \frac{\sin\left(\frac{\pi}{8}\right)}{3},
    \end{align*}
    and, hence,
    $$\int_{I_{k+1}^j} w(y)\, dy \lesssim \int_{\R} |S\chi_{I_k^j}(y)|^p w(y)\, dy \lesssim \int_{\R} \chi_{I_k^j}(y)^p w(y)\, dy = \int_{I_k^j} w(y)\, dy. $$
    Since $\frac{\sin x}{x}$ is an even function, we can interchange the roles of $I_k^j$ and $I_{k+1}^j$ to obtain the opposite inequality, so we have proved
    $$\int_{I_{k+1}^j} w(x) \, dx \approx \int_{I_{k}^j} w(x)\, dx.$$
    Adding up the equivalences for $j=1,2$, we obtain (\ref{eql}) for $l=1$.

    Let us prove (\ref{eql}) for $l>1$ with $l\not\in 4\N$. First, because it is true for $l=1$, applying (\ref{eql}) sequentially for the indices $\{k, k+1, k+2, \ldots, k+6\} \subseteq \Z$ provides
    $$ \int_{I_k} w(x) \, dx \approx \int_{I_{k+1}} w(x) \, dx \approx \int_{I_{k+2}} w(x) \, dx \approx \ldots \approx \int_{I_{k+7}} w(x) \, dx. $$
    Therefore, for any $k\in \Z$ and $l\in\{1,2,\ldots,8\}$, by adding the integrals over $I_m$ from $m=k$ to $m = k+l-1$ we have
    $$ \int_{I_{k,l}} w(x) \, dx \approx \int_{I_k} w(x) \, dx, $$
    so, by the definition of the sets $I_{k,l}^j$, we conclude that for any $k\in\Z$ and $l\geq 1$
    $$ \int_{I_{k,l}} w(x)\, dx \approx \int_{I_{k,l}^1 \cup I_{k,l}^2} w(x) \, dx = \int_{I_{k,l}^1} w(x) \, dx + \int_{ I_{k,l}^2} w(x) \, dx. $$

    Let $x\in I_{k+l,l}^j $. If $y\in I_{k,l}^j$, since $I_{n,l}^j \subseteq I_{n,l}$, we have that $x-y \leq \frac{2l\pi}{4} = l \frac{\pi}{2}$. Besides, there exist two numbers $m_x, m_y\geq 0$ such that
    \begin{align*}
        x &\in \left[ \frac{\pi}{8}(2k + 2l + 16m_x + j-1), \frac{\pi}{8}(2k + 2l + 16m_x + j) \right], \\
        y &\in \left[ \frac{\pi}{8}(2k + 16m_y + j-1), \frac{\pi}{8}(2k + 16m_y + j) \right],
    \end{align*}
    so the difference satisfies $x-y \in \frac{l}{4}\pi + 2(m_x - m_y)\pi + [ -\frac{\pi}{8}, \frac{\pi}{8} ] . $
    
    This way, whenever $l$ is not a multiple of $4$, it follows that $|\sin(x-y)| \geq \sin\left(\frac{\pi}{8}\right)$. Then,
    $$ |S\chi_{I_{k,l}^j}(x)| = \bigg| \int_{I_{k,l}^j} \frac{\sin(x-y)}{x-y} \, dy \bigg| \geq \frac{\sin\left(\frac{\pi}{8}\right)}{l\frac{\pi}{2}}|I_{k,l}^j| = \frac{\sin\left(\frac{\pi}{8}\right)}{l\frac{\pi}{2}} \left( \left\lfloor \frac{l}{8}\right\rfloor +1\right) \frac{\pi}{8} \geq \frac{\sin\left(\frac{\pi}{8}\right)}{32}. $$
    With this estimate, we get that
    $$ \int_{I_{k+l,l}^j} \hspace{-5pt} w(y)\, dy \lesssim \int_\R |S\chi_{I_{k,l}^j}(y)|^p w(y) \, dy \lesssim \int_{I_{k,l}^j} w(y)\, dy. $$
    Arguing analogously with the roles of $I_{k,l}^j$ and $I_{k+l,l}^j$ interchanged, we prove the opposite inequality. Hence, we obtain the following:
    $$\int_{I_{k+l,l}} \hspace{-5pt} w(x) \, dx \approx \int_{I_{k+l,l}^1} \hspace{-5pt} w(x) \, dx + \int_{ I_{k+l,l}^2} \hspace{-5pt} w(x) \, dx \approx \int_{I_{k,l}^1} \hspace{-3pt} w(x) \, dx + \int_{ I_{k,l}^2} w(x) \, dx \approx \int_{I_{k,l}} w(x) \, dx. $$

    Observe that the argument for the equivalence in (\ref{eql}) depends on the size of the sets involved, $I_k$, $I_k^j$, $I_{k,l}$ and $I_{k,l}^j$, and on their position relative to each other, but it does not depend on their absolute position. This means that we can repeat the proof with any translate of the previous sets to give, for every $x\in \R$, $k\in\Z$ and $l \in \N\setminus 4\N$,
    $$ \int_{x+I_{k+l,l}} \hspace{-8pt} w(y) \, dy \approx \int_{x+I_{k,l}} \hspace{-8pt} w(y) \, dy. $$

    Recall that $S$ is a self-adjoint operator, and hence it is also bounded on $L^{p'}(w^{-\frac{1}{p-1}}) $, the dual space of $L^p(w)$. Thus, we can make the same arguments and obtain the equivalence (\ref{eql}) for $w^{-\frac{1}{p-1}}$ in place of $w$.

    Let us now prove the following estimate for each $k \in\Z$ and $l\in 2\N \setminus 4\N$:
    $$ \bigg( \int_{I_{k+l,l}} w(x) \, dx \bigg) \bigg( \int_{I_{k,l}} w(x)^{-\frac{1}{p-1}} \, dx\bigg)^{p-1} \lesssim |I_{k,l}|^p.  $$
    Since we can compare each integral over $I_{n,l}$ to that over $I_{n,l}^1 \cup I_{n,l}^2 $, for each $n\in\Z$, we need only show that
    $$ \bigg( \int_{(J_{k,l})+\frac{\pi}{4}l} w(x) \, dx \bigg) \bigg( \int_{J_{k,l}} w(x)^{-\frac{1}{p-1}} \, dx \bigg)^{p-1} \lesssim |I_{k,l}|^p, $$
    where we denote $J_{k,l} = I_{k,l}^1 \cup I_{k,l}^2$.

    Let $f = w^{-\frac{1}{p-1}} \chi_{J_{k,l}}$, and let $x\in J_{k,l} + \frac{\pi}{4}l$ and $y\in J_{k,l}$. Then, we can argue as before: first, the difference must satisfy $|x-y| \leq \frac{l\pi}{2}.$ Besides this, there exist $j_x,j_y\in \{1,2\}$ and $m_x, m_y\in \Z$ such that
    \begin{align*}
        x &\in \left[ \frac{\pi}{8}(2k + 2l + 16m_x + j_x-1), \frac{\pi}{8}(2k + 2l + 16m_x + j_x) \right], \\
        y &\in \left[ \frac{\pi}{8}(2k + 16m_y + j_y-1), \frac{\pi}{8}(2k + 16m_y + j_y) \right],
    \end{align*}
    so the difference satisfies $x-y \in \frac{l}{4}\pi + 2(m_x - m_y)\pi + [ -\frac{\pi}{4}, \frac{\pi}{4}].$
    
    Thus, whenever $l \in 2\N \setminus 4\N$, this interval is centered at some $n\pi + \frac{\pi}{2}$, so $|\sin(x-y)| \geq \sin(\frac{\pi}{4})$. In this case, we get    
    $$ |Sf(x)| = \int_{J_{k,l}} \left |\frac{\sin(x-y)}{x-y}\right | w(y)^{-\frac{1}{p-1}} \, dy \gtrsim \frac{1}{|I_{k,l}|} \int_{J_{k,l}} w(y)^{-\frac{1}{p-1}} \, dy, $$
    so we can finally see that
    \begin{align*}
        \bigg( \frac{1}{|I_{k,l}|} \int_{J_{k,l}} w(y)^{-\frac{1}{p-1}} \, dy \bigg)^p \bigg( \int_{J_{k+l,l}} w(y) \, dy \bigg) &\lesssim \int_{\R} |Sf(y)|^p \chi_{J_{k+l,l}} w(y) \, dy \\
        &\lesssim \int_{J_{k,l}} w(y)^{-\frac{p}{p-1}} w(y) \, dy =  \int_{J_{k,l}} w(y)^{-\frac{1}{p-1}} \, dy.
    \end{align*}
    Again, note that this reasoning does not depend on the absolute position of the intervals, so we can argue the same for any translates of the intervals involved.

    Now, let $I$ be an arbitrary interval with $|I| > \pi$. Let $y\in \R$, $k\in\Z$ and $l\geq 4$ such that $y + I_{k,l} \subseteq I \subseteq y+ I_{k,l+1}$. Note that in the set $\{l-3, l-2, l-1, l\}$ there is one number that also lies in $2\N \setminus 4\N$. We can see that $I_{k,l+1} \subseteq 5 I_{k,l-3}$ for any $l\geq 4$. Thus, for any $I$ of length greater than $\pi$, there exists an $l\in 2\N \setminus 4\N$ such that $y+ I_{k,l} \subseteq I \subseteq 5(y+I_{k,l}). $ Therefore, since $5(y + I_{k,l}) = y + I_{k-2l,5l}$ and $5l\in 2\N \setminus 4 \N$, we obtain that
    \begin{align*}
        \bigg( \int_I w(x) \, dx \bigg) \left( \int_I w(x)^{-\frac{1}{p-1}} \, dx \right)^{p-1} &\leq \bigg( \int_{5(y+I_{k,l})} w(x) \, dx \bigg) \left( \int_{5(y+I_{k,l})} w(x)^{-\frac{1}{p-1}} \, dx \right)^{p-1} \\
        &\lesssim |5(y+I_{k,l})|^p \lesssim |y + I_{k,l}|^p \leq |I|^p,
    \end{align*}
    so we have $w\in A_{p,1}$.

    Let us prove the sufficiency of $w\in A_{p,1}$. Recall that we have the following equality
    $$\frac{\sin(x-y)}{x-y} f(y) = \sin x \frac{\cos y f(y)}{x-y} - \cos x \frac{\sin y f(y)}{x-y} = \sin x \frac{f_c(y)}{x-y} - \cos x \frac{f_s(y)}{x-y},$$
    with $f_c = \cos \cdot f$ and $f_s = \sin \cdot f$. With this, we can express the segment multiplier as
    \begin{align*}
        Sf(x) &= \int_\R \frac{\sin(x-y)}{x-y} f(y) \, dy = \int_{|x-y| > 2\pi} \frac{\sin(x-y)}{x-y} f(y) \, dy + \int_{|x-y| \leq 2\pi} \frac{\sin(x-y)}{x-y} f(y) \, dy \\
        &= \sin x H_{2\pi}f_c(x) - \cos x H_{2\pi}f_s(x) + \int_{|x-y| \leq 2\pi} \frac{\sin(x-y)}{x-y} f(y) \, dy,
    \end{align*}
    so, taking into account that $B_{2\pi}$ is a bounded operator on $L^p(w)$ (see Proposition~\ref{unifbdd}) and that
    $$\left| \int_{|x-y| \leq 2\pi} \frac{\sin(x-y)}{x-y} f(y) \, dy \right| \leq \int_{|x-y| \leq 2\pi} |f(y)| \, dy = 4\pi B_{2\pi} (|f|)(x),$$
    we can bound the multiplier as
    $$|Sf| \lesssim |H_{2\pi}f_c| + |H_{2\pi}f_s| + B_{2\pi}(|f|). $$
    Because of Proposition~\ref{unifbdd}, it is sufficient to check that the truncated Hilbert transform is bounded on $L^p(w)$.

    Let us prove that $H_{2\pi}$ is bounded. Let $g\in L^p(w)$ and $x\in\R$. Put $k = \lfloor\frac{x}{\pi}\rfloor$, so that $x \in [k\pi, (k+1)\pi)$. Then, we have
    \begin{align*}
        H_{2\pi}g(x) &= \sum_{|k-j| \geq 3} \int_{j\pi}^{(j+1)\pi} \frac{g(y)}{x-y} \, dy + \int_{(k-2)\pi}^{x-2\pi} \frac{g(y)}{x-y} \, dy + \int_{x+2\pi}^{(k+3)\pi} \frac{g(y)}{x-y} \, dy \\
        &= \frac{1}{\pi} \sum_{|k-j|\geq 3} \frac{\int_{j\pi}^{(j+1)\pi} g(y) \, dy}{k-j} + \sum_{|k-j|\geq 3} \int_{j\pi}^{(j+1)\pi} \left( \frac{1}{x-y} - \frac{1}{k\pi - j\pi} \right) g(y) \, dy \\
        &\quad + \int_{(k-2)\pi}^{x-2\pi} \frac{g(y)}{x-y} \, dy + \int_{x+2\pi}^{(k+3)\pi} \frac{g(y)}{x-y} \, dy.
    \end{align*}
    
    The last two integrals in the right-hand side can be bounded as
    $$\bigg| \int_{(k-2)\pi}^{x-2\pi} \frac{g(y)}{x-y} \, dy + \int_{x+2\pi}^{(k+3)\pi} \frac{g(y)}{x-y} \, dy \bigg| \leq \int_{x-3\pi}^{x+3\pi} \frac{|g(y)|}{2\pi} \, dy = 3 B_{3\pi}(|g|)(x).  $$
    
    For the second term, note that if $x\in [k\pi, (k+1)\pi)$ and $y \in [j\pi, (j+1)\pi)$, then the difference satisfies $ x-y \in [(k-j-1)\pi, (k-j+1)\pi], $ or, put in another way, we have $ (k-j)\pi \in [x-y -\pi, x-y +\pi], $ so it follows that
    $$ \frac{1}{(k-j)\pi} \in \left[ \frac{1}{x-y + \pi}, \frac{1}{x-y - \pi} \right]. $$
    With this, and with the condition $|k-j|\geq 3$, we have the following inequality:
    $$ \left| \frac{1}{x-y} - \frac{1}{k\pi - j\pi} \right| \lesssim \frac{1}{1 + (x-y)^2}. $$
    Thus, the second term can be bounded as
    \begin{align*}
        \bigg| \sum_{|k-j|\geq 3} \int_{j\pi}^{(j+1)\pi} \left( \frac{1}{x-y} - \frac{1}{k\pi - j\pi} \right) g(y) \, dy \bigg| &\lesssim \sum_{|k-j|\geq 3} \int_{j\pi}^{(j+1)\pi} \frac{ |g(y)|}{1+ (x-y)^2} \, dy \\ 
        &\leq \sum_{j\in \Z} \int_{j\pi}^{(j+1)\pi} \frac{ |g(y)|}{1+ (x-y)^2} \, dy \\
        &= \int_\R \frac{|g(y)|}{1 + (x-y)^2} \, dy = P(|g|)(x),
    \end{align*}
    where $P$ is the operator introduced in Definition \ref{def:poisson}. Considering both Proposition~\ref{unifbdd} and Lemma~\ref{poisson}, we know that $P$ is bounded in $L^p(w)$.

    Hence, to prove the boundedness of $H_{2\pi}$, we only need to show that
    $$ \Tilde{K}(g)(x) = \sum_{|\lfloor\frac{x}{\pi}\rfloor -j|\geq 3} \frac{\int_{j\pi}^{(j+1)\pi} g(y) \, dy}{\lfloor\frac{x}{\pi}\rfloor -j} = \sum_{|k-j|\geq 3} \frac{(Vg)_j}{k-j} =: (\H_3 V g)_{k}, \quad x\in [k\pi, (k+1)\pi), $$
    is a bounded operator on $L^p(w)$, where $\H_3$ denotes the corresponding truncated discrete Hilbert transform. It is clear that $\|\Tilde{K} g\|_{L^p(w)} = \|\H_3 V g\|_{\ell^p(w_d)}$, it can be immediately checked that $\H - \H_3$ is bounded on $\ell^p(w_d)$ and, because of Proposition~\ref{discrap}, we know that $w_d\in A_p^d$ and $\H$ is bounded as well. Then,
    $$ \|\Tilde{K}g\|_{L^p(w)} \leq \|(\H_3 - \H) V g\|_{\ell^p(w_d)} + \|\H V g\|_{\ell^p(w_d)} \leq (\|\H - \H_3\| + \|\H\|) \|V\| \|g\|_{L^p(w)}. $$
    Therefore, the boundedness for $\Tilde{K}$ follows and, consequently, for $H_{2\pi}$ and $S$, finishing the proof.
\end{proof}

\begin{rem}
    As shown in Proposition \ref{ex:ap1counter}, there are weights that satisfy the $A_{p,1}$ condition while failing to be in $A_p$. Because of Theorem \ref{thm:main}, this means that there are weighted spaces $L^p(w)$ where the segment multiplier is bounded but the Hilbert transform is not.
\end{rem}

\begin{rem}
    A class similar in definition to $A_{p,1}$ is that of local Muckenhoupt weights; that is, weights $w$ such that
    $$ [w]_{A_p^{\text{loc}}} = \sup_{|I|\leq 1} \left( \frac{1}{|I|} \int_I w(x) \, dx \right) \left(\frac{1}{|I|} \int_I w(x)^{-\frac{1}{p-1}}\, dx \right)^{p-1} < \infty. $$
    Note that now the supremum covers all small intervals, $|I| \leq 1$, while it misses the large ones. This condition was first introduced by Rychkov in \cite{rychkov}, and it satisfies analogous properties to the ones we have proved for $A_{p,1}$. Both classes are distinct, and it is clear from their definitions that $A_{p,1} \cap A_p^{\text{loc}} = A_p$.
\end{rem}

\end{document}